\documentclass{amsart}

\usepackage{setspace}

\usepackage[all]{xy}
\usepackage{tikz,tikz-cd}
\usetikzlibrary{shapes, matrix, arrows}
\usepackage{adjustbox}
\usepackage{eqparbox}
\makeatletter

\usepackage{subcaption}

\newcommand{\rvline}{\hspace*{-\arraycolsep}\vline\hspace*{-\arraycolsep}}

\usepackage{graphicx}
\usepackage{multirow}
\usepackage{amsmath,amssymb,amsfonts}
\usepackage{amsthm}
\usepackage{mathrsfs}
\usepackage[title]{appendix}
\usepackage{xcolor}
\usepackage{textcomp}
\usepackage{manyfoot}
\usepackage{booktabs}
\usepackage{algorithm}
\usepackage{algorithmicx}
\usepackage{algpseudocode}
\usepackage{listings}
\usepackage{tabstackengine}

\theoremstyle{thmstyleone}
\newtheorem{theorem}{Theorem}
\newtheorem{proposition}[theorem]{Proposition}

\theoremstyle{thmstyletwo}

\newtheorem{remark}{Remark}

\theoremstyle{thmstylethree}

\newcommand\Z{\mathbb{Z}}
\newcommand\Q{\mathbb{Q}}
\newcommand\R{\mathbb{R}}

\copyrightinfo{2023}{Machiel van Frankenhuijsen, Edward K. Voskanian}

\begin{document}

\title[]{A note on the quality of simultaneous Diophantine approximations obtained by the LLL algorithm}

\author{Machiel van Frankenhuijsen}
\address{Department of Mathematics, Utah Valley University \newline  
\indent Orem, Utah 84058, USA}
\email{vanframa@uvu.edu}

\author{Edward K. Voskanian}
\address{Department of Mathematics, Norwich University \newline 
\indent Northfield, Vermont 05663, USA}
\email{evoskani@norwich.edu}

\keywords{Diophantine approximation, simultaneous Diophantine approximation, LLL Algorithm, algorithm of Lenstra, Lenstra and Lov\'asz.}

\begin{abstract}
In 1982, A. K. Lenstra, H. W. Lenstra, and L. Lovász introduced the first polynomial-time method to factor a nonzero polynomial $f \in \mathbb{Q}[x]$ into irreducible factors. This algorithm, now commonly referred to as the LLL Algorithm, can also be applied to compute simultaneous Diophantine approximations. We present a significant improvement of a result by Bosma and Smeets on the quality of simultaneous Diophantine approximations achieved by the LLL Algorithm.
\end{abstract}

\maketitle

\section{Preliminaries} \label{sec:preliminaries}

\subsection{Simultaneous Diophantine Approximation} \label{SDA}
\hfill \\
The continued fraction of an irrational number $a$ produces infinitely many rational numbers $p/q$ such that
\[ |qa-p|<\frac1q. \]
Using $\|x\|$ to denote the distance of the real number $x$ to the nearest integer, the above result can be expressed as
\[ q\|qa\|<1. \]
Dirichlet proved the following theorem about simultaneously approximating several real numbers.
\begin{theorem}\label{T: max}
Let\/ $A$ be an $n\times m$ matrix with real entries such that the first row together with\/~$1$ is independent over\/ $\Q,$
\[ [1\ a_{11}\ \dots\ a_{1m}]\quad \text{ is independent over\/ }\Q. \]
Then there exist infinitely many sets of coprime integers\/ $(q_1,\dots, q_m)$ such that
\[ q=\max_{j\leq m}|q_j|\geq1 \]
and
\[ q^{m/n}\biggl\|\sum_{j=1}^ma_{ij}q_j\biggr\|<1\quad\text{ for all\/ }i\leq n. \]
\end{theorem}
If $m=1$,
then the theorem finds one common denominator to simultaneously approximate $n$ real numbers,
\[ \|qa_i\|<q^{-1/n}\quad\text{ for all }i\leq n. \]

Next, we give a brief survey of the LLL Algorithm \cite{lenstra1982factoring}. For a more comprehensive understanding of the LLL Algorithm and the associated technique of lattice basis reduction, see, e.g., \cite{bremner2011lattice} which offers a beginner-friendly explanation of the fundamentals of the LLL Algorithm.

\subsection{The LLL Algorithm} \label{LLL}
\hfill \\
A basis $\vec{b}_1, \dots, \vec{b}_n$ of $\R^n$ spans a lattice of integer points in $\R^n$,
\[ L = \{r_1\vec{b_1} + \dots + r_n\vec{b_n}\colon r_i\in\Z\text{ for }1\leq i\leq n\}. \]
The determinant of $L$ is the absolute value of the determinant of the matrix with columns $\vec{b}_i$,
\[ \det(L) := \lvert \det(\vec{b_1} \dots \vec{b_n}) \rvert. \]
This matrix can be reduced in the sense that the basis can be chosen to be almost orthogonal. Indeed,
\[ \det(L) \leq \lvert \vec{b}_1 \rvert \cdot \lvert \vec{b}_2 \rvert \cdots \lvert \vec{b}_n \rvert, \]
with equality if and only if the basis is orthogonal.

Recall the Gramm-Schmidt process: for $i = 1, 2, \dots, n$, compute the component $\vec{b}_i^\ast$ of $\vec{b}_i$ that is
orthogonal to the previous $\vec{b}_1^\ast, \dots, \vec{b}_{i-1}^\ast$,
\[ \vec{b}_i^\ast = \vec{b}_i - \mu_{i,1}\vec{b}_1^* - \mu_{i,2}\vec{b}_2^* - \cdots - \mu_{i,i-1}\vec{b}_{i-1}^*, \]	
where
\[ \mu_{i,j} = \frac{\langle \vec{b}_i, \vec{b}_j^\ast \rangle}{\langle \vec{b}_j^\ast, \vec{b}_j^\ast \rangle}. \]
Then $\vec{b}_1^\ast, \dots, \vec{b}_n^\ast$ is an orthogonal basis for $\R^n$, and at each step, $\vec{b}_1^\ast, \dots, \vec{b}_i^\ast$ is an orthogonal basis for $\mathbb{R}\vec{b}_1 + \dots + \mathbb{R}\vec{b}_i$.

The LLL Algorithm returns a basis $\vec{b}_1, \dots, \vec{b}_n$ that is almost orthogonal in the sense that
\[ \lvert \mu_{i,j} \rvert \leq \frac{1}{2}\quad \text{ for } 1 \leq j < i \leq n, \]
and
\[ \left\lvert \vec{b}_i^\ast + \mu_{i,i-1}\vec{b}_{i-1}^\ast \right\rvert^2 \geq \alpha\left\lvert \vec{b}_{i-1}^\ast \right\rvert^2\quad \text{ for } 2 \leq i \leq n, \]
where $1/4 < \alpha < 1$ is called the {\it reduction parameter}. After running LLL with $1/4 < \alpha < 1$, the result is a basis such that
\begin{gather}
\left\lvert \vec{b}_j \right\rvert^2 \leq \beta^{i-1}\left\lvert \vec{b}_i^* \right\rvert^2\quad \text{ for } 1 \leq j \leq i \leq n, \\
\det(L) \leq \lvert \vec{b}_1 \rvert \cdots \lvert \vec{b}_n \rvert \leq \beta^{n(n-1)/4}\det(L), \\
\lvert \vec{b}_1 \rvert \leq \beta^{(n-1)/4}\det(L)^{1/n},
\end{gather}
where $\beta := 4/(4\alpha - 1)$.

\subsection{LLL Algorithm for Simultaneous Diophantine Approximation} \label{sec:LLL algorithm for simultaneous Diophantine approximation}

In \cite{bosma2013finding}, the application of the original LLL Algorithm to obtain ``efficient'' rational approximations, with a common denominator, to the real numbers $a_1, a_2, \dots, a_n$ \cite[Proposition 1.39]{lenstra1982factoring} (stated below as Proposition \ref{original simultaneous approximation by LLL}), was extended to approximate (in a precise sense that is explained below) the $n \times m$ matrix of real numbers  
\begin{equation} \label{target}
\setstackgap{L}{1.1\baselineskip}
\fixTABwidth{T}
A = \parenMatrixstack{
a_{11} & a_{12} & \cdots & a_{1m}  \\
a_{21} & a_{22} & \cdots & a_{2m}  \\
\vdots & \vdots & \ddots & \vdots \\
a_{n1} & a_{n2} & \cdots & a_{nm}
}.
\end{equation}
Specifically, as shown in \cite[Lemma 2.3]{bosma2013finding}, for a given {\it target matrix} (\ref{target}) and real number $0 < \varepsilon < 1$, the LLL Algorithm applied to a lattice in $\mathbb{R}^{n + m}$ with basis given by the columns of the following $(n + m)$-square matrix
\begin{equation} \label{basis matrix} 
B_{A,c} =
\begin{pmatrix}
  \begin{matrix}
  I_n \\
  \end{matrix}
  & \rvline & A \\
\hline
  O & \rvline &
  \begin{matrix}
  cI_{m}
  \end{matrix}
\end{pmatrix},
\end{equation}
where 
\begin{equation} \label{general c}
c = \beta^{\frac{(1 - m - n)(n + m)}{4m}}\varepsilon^{1 + \frac{n}{m}},
\end{equation}
yields an $m$-tuple $(q_1, q_2, \dots, q_m)$ of integers satisfying
\begin{equation} \label{approximation bound with beta}
\displaystyle\max_{1 \leq j \leq m}\vert q_j \vert \leq \beta^{\frac{(m + n - 1)(m + n)}{4m}}\varepsilon^{-\frac{n}{m}}
\end{equation} 
and
\begin{equation} \label{approximation bound}
\max_{1 \leq i \leq n}\| q_1a_{i1} + q_2a_{i2} + \cdots + q_ma_{im} \| \leq \varepsilon.
\end{equation}

\begin{remark}
In \cite[Lemma 2.3]{bosma2013finding}, the authors apply LLL with $\alpha = 3/4$ which make $\beta = 2$. After applying LLL with this reduction parameter, inequality \textup(\ref{approximation bound with beta}\textup) becomes
\begin{equation*}
\displaystyle\max_{1 \leq j \leq m}\vert q_j \vert \leq 2^{\frac{(m + n - 1)(m + n)}{4m}}\varepsilon^{-\frac{n}{m}}
\end{equation*} 
by replacing \textup(\ref{general c}\textup) with  
\[ c = 2^{\frac{(1 - m - n)(n + m)}{4m}}\varepsilon^{1 + \frac{n}{m}}. \]
\end{remark}

If we set $m = 1$ in the derivation above, we obtain the original application of the LLL Algorithm to the problem of simultaneously approximating two or more real numbers with rational numbers having a common denominator:
 
\begin{proposition} \label{original simultaneous approximation by LLL}
There exists a polynomial-time algorithm that, given a positive integer $n$, rational numbers $a_1, a_2, \dots, a_n, \varepsilon$ satisfying $0 < \varepsilon < 1$, and reduction parameter $1/4 < \alpha < 1$, finds integers $p_1, p_2, \dots, p_n, q$ for which
\[ \vert p_i - qa_i \vert \leq \varepsilon \text{ for } 1 \leq i \leq n, \]
\[ 1 \leq q \leq \beta^{\frac{n(n + 1)}{4}}\varepsilon^{-n}. \]
\end{proposition}

In \cite{bosma2013finding}, the authors present an algorithm that generates finitely many nondecreasingly good rational approximations by iterating the LLL algorithm, which they call the Iterated LLL (ILLL) Algorithm. After each application of the LLL algorithm, the value of $c$ is decreased before it is applied again, giving another approximation that is the same or better than the previous one. The ILLL Algorithm, as presented in \cite[Section 3]{bosma2013finding}, specifies the reduction parameter $\alpha = 3/4$, $\varepsilon = 1/2$, and the amount that $c$ is decreased by after each iteration. In the next section, we restate their algorithm in more generality. Furthermore, in Theorem \ref{main theorem}, we present the corresponding version of \cite[Theorem 3.5]{bosma2013finding} on the guaranteed quality of approximations from the ILLL Algorithm, which when restricted to the special case of the ILLL Algorithm in \cite[Section 3]{bosma2013finding}, gives a smaller bound on the error of approximations.

\section{Main Result} \label{sec:main result}

In this section, we begin with a survey of the Iterated LLL (ILLL) Algorithm (with $\alpha = 3/4$) by Bosma and Smeets \cite[Section 3]{bosma2013finding}, which by iterating the original LLL Algorithm, generates a finite sequence of approximations to a target matrix $A$. Then, we present an improvement of their main result concerning the efficiency of the ILLL Algorithm. Note that our survey of the ILLL Algorithm and improvement of \cite[Theorem 3.5]{bosma2013finding} leaves the reduction parameter $1/4 < \alpha < 1$ unspecified in the interest of generality.

\begin{algorithm}
\caption{ILLL Algorithm \cite[Algorithm 3.1]{bosma2013finding}}
\begin{algorithmic}[1]
\Require A target matrix $A$ as given in (\ref{target}), real numbers $\varepsilon, d$ satisfying $0 < \varepsilon < 1 < d$, a choice of reduction parameter $1/4 < \alpha < 1$, and an upper bound $q_{\text{max}} > 1$.
\Ensure For each integer $1 \leq k \leq k'$ ($k'$ as in (\ref{k'}) below), an $m$-tuple $q(k) = (q_1(k), q_2(k), \dots, q_m(k))$ satisfying
\begin{equation} \label{k denominator bound}
\max_{1 \leq j \leq m}\vert q_j(k) \vert \leq \beta^{\frac{(m + n - 1)(m + n)}{4m}}(\varepsilon d^{1 - k})^{-\frac{n}{m}},
\end{equation}
and
\begin{equation} \label{k approximation bound}
\max_{1 \leq i \leq n}\| q_1(k)a_{i1} + q_2(k)a_{i2} + \cdots + q_m(k)a_{im} \| \leq \varepsilon d^{1 - k}. 
\end{equation}
\State Construct the basis matrix $B_{A,c(\varepsilon)}$ as given in (\ref{basis matrix}) from the target matrix $A$.
\State Apply the LLL Algorithm to $B_{A,c(\varepsilon)}$. \label{step:apply LLL} 
\State Deduce the $m$-tuple $(q_1, q_2, \dots, q_m)$. \label{step:deduce denominators}
\State Replace $c$ by $cd^{-\frac{n + m}{m}}$. \label{step:replace}
	\If{$\displaystyle\max_{1 \leq j \leq m}\vert q_j \vert \leq q_{\text{max}}$}
            \State go to Step \ref{step:apply LLL} \label{step:go to apply LLL}
         \Else
            \State stop
         \EndIf
\end{algorithmic}
\label{alg:ILLL}
\end{algorithm}

In the first iteration of the ILLL Algorithm (see Algorithm \ref{alg:ILLL}), after Step 3, the $m$-tuple $q(1) = (q_1(1), q_2(1), \dots, q_m(1))$ satisfies inequalities (\ref{approximation bound with beta}) and (\ref{approximation bound}) \cite[Lemma 3.4]{bosma2013finding}. After Step 4, the basis matrix $B_{A,c(\varepsilon)}$ is updated with $B_{A,c(\varepsilon d^{-1})}$. Then, if Step \ref{step:apply LLL} is reached, Step \ref{step:deduce denominators} in the second iteration yields the $m$-tuple $q(2) = (q_1(2), q_2(2), \dots, q_m(2))$ that satisfies
\[ \max_{1 \leq j \leq m}\vert q_j(2) \vert \leq \beta^{\frac{(m + n - 1)(m + n)}{4m}}(\varepsilon d^{-1})^{-\frac{n}{m}}, \]
and
\[ \max_{1 \leq i \leq n}\| q_1(2)a_{i1} + q_2(2)a_{i2} + \cdots + q_m(2)a_{im} \| \leq \varepsilon d^{-1}. \] 
Therefore, after Step 4 in the $(k - 1)\text{st}$ iteration, the basis matrix is updated to $B_{A,c(\varepsilon d^{1 - k})}$ so that if Step \ref{step:apply LLL} is reached, Step \ref{step:deduce denominators} in the $k\text{th}$ iteration yields the $m$-tuple $q(k) = (q_1(k), q_2(k), \dots, q_m(k))$ that satisfies inequalities (\ref{k denominator bound}) and (\ref{k approximation bound}).

The number of times the ILLL Algorithm calls upon the LLL Algorithm is obtained by solving the inequality
\[ q_{\text{max}} \leq \beta^{\frac{(m + n - 1)(m + n)}{4m}}(\varepsilon d^{1 - k})^{-\frac{n}{m}} \]
for $k$:
\begin{gather*}
\beta^{-\frac{(m + n - 1)(m + n)}{4m}}q_{\text{max}} \leq (\varepsilon d)^{-\frac{n}{m}}d^{\frac{kn}{m}} \\
\beta^{-\frac{(m + n - 1)(m + n)}{4m}}(\varepsilon d)^{\frac{n}{m}}q_{\text{max}} \leq d^{\frac{kn}{m}} \\
\frac{m}{n}\log_d(q_{\text{max}}) + 1 + \log_d(\varepsilon) - \frac{(m + n - 1)(m + n)}{4n}\log_d(\beta) \leq k.
\end{gather*}

If in the $k\text{th}$ iteration, $\max_{1 \leq j \leq m}\vert q_j(k) \vert > \beta^{\frac{(m + n - 1)(m + n)}{4m}}(\varepsilon d^{1 - k})^{-\frac{n}{m}}$, the algorithm stops. Therefore, the algorithm stops after 
\begin{equation} \label{k'} k' = \left\lceil \frac{m}{n}\log_d(q_{\text{max}}) + 1 + \log_d(\varepsilon) - \frac{(m + n - 1)(m + n)}{4n}\log_d(\beta) \right\rceil
\end{equation}
iterations.

\begin{remark}
To ease notation, the authors of \cite[Algorithm 3.1]{bosma2013finding} set $\varepsilon = \frac{1}{d}$ which turns (\ref{k denominator bound}) and (\ref{k approximation bound}) into
\begin{equation}
\max_{1 \leq j \leq m}\vert q_j(k) \vert \leq \beta^{\frac{(m + n - 1)(m + n)}{4m}}d^{\frac{kn}{m}},
\end{equation}
and
\begin{equation}
\max_{1 \leq i \leq n}\| q_1(k)a_{i1} + q_2(k)a_{i2} + \cdots + q_m(k)a_{im} \| \leq d^{-k}, 
\end{equation}
respectively \textup(see \cite[Remark 3.2]{bosma2013finding}\textup). Furthermore, by \cite[Lemma 3.3]{bosma2013finding}, putting $d = 2$ makes  
\[ k' = \left\lceil \frac{m\log_2 q_{\text{max}}}{n} - \frac{(m + n - 1)(m + n)}{4n}\log_2(\beta) \right\rceil. \]
\end{remark}

We now conclude with the main result of the present paper.

\begin{theorem} \label{main theorem}
Let an $n \times m$-matrix $A$ with entries $a_{ij}$ in $\mathbb{R}$. The \textup{ILLL} algorithm with reduction parameter $1/4 < \alpha < 1$, and $0 < \varepsilon < 1 < d$, finds a sequence of $m$-tuples $(q_1, \dots, q_m)$ of integers such that for every $Q$ with $\varepsilon^{-\frac{n}{m}}\beta^{\frac{(m + n - 1)(m + n)}{4m}} \leq Q \leq q_{\textup{max}}$, one of these $m$-tuples satisfies
\begin{equation}
\max_{1 \leq j \leq m}\lvert q_j \rvert \leq Q,\quad \textup{and}
\end{equation}
\begin{equation} \label{bound2}
\max_{1 \leq i \leq n}\| q_1a_{i1} + \cdots + q_ma_{im} \| \leq dQ^{-\frac{m}{n}}\beta^{\frac{(m + n - 1)(m + n)}{4n}}.
\end{equation} 
\end{theorem}
\begin{proof}
Take $k \in \mathbb{N}$ such that
\begin{equation} \label{ineq1}
d^{(k - 1)\frac{n}{m}} \leq Q\varepsilon^{\frac{n}{m}}\beta^{\frac{(1 - m - n)(m + n)}{4m}}< d^{k\frac{n}{m}}.
\end{equation}
By our assumption that $Q \leq q_{\textup{max}}$ and by the left-hand side of (\ref{ineq1}), we are guaranteed that $k \leq k'$ so that we may invoke \cite[Lemma 3.4]{bosma2013finding} to assert that $q(k) = (q_1(k), \dots, q_m(k))$ satisfies the inequality
\[ \max_{1 \leq j \leq m}\vert q_j(k) \vert \leq \beta^{\frac{(m + n - 1)(m + n)}{4m}}(\varepsilon d^{1 - k})^{-\frac{n}{m}} \leq Q. \]
From the right-hand side of (\ref{ineq1}) it follows that
\begin{equation}\label{ineq2}
\varepsilon d^{1 - k} < dQ^{-\frac{m}{n}}\beta^{\frac{(m + n - 1)(m + n)}{4n}}.
\end{equation}
From (\ref{ineq2}) and \cite[Lemma 3.4]{bosma2013finding} we know that
\[ \max_{1 \leq i \leq n}\| q_1a_{i1} + \cdots + q_ma_{im} \| \leq \varepsilon d^{1 - k} < dQ^{-\frac{m}{n}}\beta^{\frac{(m + n - 1)(m + n)}{4n}}. \]  
\end{proof}

We note that \cite[Theorem 3.5]{bosma2013finding} concerns the special case of the ILLL Algorithm with $\varepsilon = 1/2$ (so that $\beta = 2$) and $d = 2$, which gives the error bound  
\begin{align*}
\max_{1 \leq i \leq n}\| q_1a_{i1} + \cdots + q_ma_{im} \| &\leq Q^{-\frac{m}{n}}2^{\frac{(m + n + 3)(m + n)}{4n}} \\
&= Q^{-\frac{m}{n}}2^{\frac{(m + n - 1)(m + n)}{4n} + 1}2^{\frac{m}{n}}.
\end{align*}
When restricted to this case, the error bound in (\ref{bound2}) is smaller than the one above by the factor $2^{\frac{m}{n}}$. Indeed, Theorem \ref{main theorem} is a significant improvement of \cite[Theorem 3.5]{bosma2013finding}. 



\end{document}